\documentclass[11pt]{amsart}

\usepackage{amssymb}

\newcommand{\cal}{\mathcal}
\newcommand{\mf}{\mathfrak}

\newcommand{\lb}{\left\{}
\newcommand{\rb}{\right\}}
\newcommand{\la}{\left<}
\newcommand{\ra}{\right>}

\newcommand{\zfc}{\mathsf{ZFC}}

\newcommand{\ps}{\mathbb{P}}

\newcommand{\al}{\alpha}
\newcommand{\be}{\beta}
\newcommand{\ga}{\gamma}

\newcommand{\de}{\delta}

\newcommand{\ka}{\kappa}
\newcommand{\lam}{\lambda}

\newcommand{\vp}{\varphi}
\newcommand{\om}{\omega}

\newcommand{\bsl}{\setminus}
\newcommand{\lra}{\longrightarrow}
\newcommand{\res}{\upharpoonright}
\newcommand{\seq}{\subseteq}
\newcommand{\we}{\wedge}

\newcommand{\cf}{\operatorname{cf}}
\newcommand{\pcf}{\operatorname{pcf}}
\newcommand{\spec}{\operatorname{spec}}

\newcommand{\es}{\emptyset}
\newcommand{\U}{\mathcal{U}}

\theoremstyle{definition}
\newtheorem{definition}{Definition}[section]

\newtheorem{question}[definition]{Question}

\newtheorem{notation}[definition]{Notation}

\theoremstyle{plain}
\newtheorem{fact}[definition]{Fact}
\newtheorem{theorem}[definition]{Theorem}
\newtheorem{proposition}[definition]{Proposition}
\newtheorem{lemma}[definition]{Lemma}
\newtheorem{corollary}[definition]{Corollary}

\newtheorem{remark}[definition]{Remark}

\usepackage[pdftex]{hyperref}
\hypersetup{
    colorlinks=true, 
    linktoc=all,     
    linkcolor=blue,  
}

\title{PCF Theory and The Tukey Spectrum}
\author{Thomas Gilton}

\begin{document}

\address{University of Pittsburgh
Department of Mathematics.
The Dietrich School of 
Arts and Sciences,
301 Thackeray Hall,
Pittsburgh, PA 15260, United States}
\email{tdg25@pitt.edu}

\date{\today}



\begin{abstract}
    In this paper, we investigate the relationship between the Tukey order and PCF theory, as applied to sets of regular cardinals. We show that it is consistent that for all sets $A$ of regular cardinals that the Tukey spectrum of $A$, denoted $\spec(A)$, is equal to the set of possible cofinalities of $A$, denoted $\pcf(A)$; this is to be read in light of the $\mathsf{ZFC}$ fact that $\pcf(A)\seq\spec(A)$ holds for all $A$. We also prove results about when regular limit cardinals must be in the Tukey spectrum or must be out of the Tukey spectrum of some $A$, and we show the relevance of these for forcings which might separate $\spec(A)$ from $\pcf(A)$. Finally, we show that the \emph{strong part} of the Tukey spectrum can be used in place of PCF-theoretic scales to lift the existence of J{\'o}nsson algebras from below a singular to hold at its successor. We close with a list of questions.
\end{abstract}

\maketitle

\section{Introduction}

The Tukey order has become a very useful tool for comparing directed, partially ordered sets. The order is sufficiently coarse that is allows us to compare many different partial orders, yet it is fine enough to preserve a number of order-theoretic properties of interest (such as calibre properties; see Proposition \ref{prop:TukeyCalibre} below). The Tukey order works by comparing partial orders in terms of what happens ``eventually", or more precisely, in terms of what happens \emph{cofinally}. 

The Tukey order arose in the study of Moore-Smith convergence in topology  (\cite{MooreSmith} and \cite{Tukey}), with \cite{Birkhoff} and \cite{Day} following shortly after. Schmidt (\cite{Schmidt}) and Isbell (\cite{IsbellCategory} and \cite{IsbellCofinal}) later studied cofinal types among the class of directed posets. Later, Todor{\v c}evi{\' c} (\cite{TodorcevicDirected}) showed that it is consistent that there are only five cofinal types of directed sets of size $\leq\aleph_1$. Todor{\v c}evi{\' c} also showed that under the $\mathsf{CH}$, there are $2^{\mf{c}}$-many such cofinal types, and in \cite{TodorcevicClassification} he extended this result to all transitive relations on $\omega_1$.

Since then, there has been a tremendous amount of research on the Tukey order in a variety of circumstances. Some results concern definable directed sets (\cite{SoleckiTodorcevicAvoiding}, \cite{SoleckiTodorcevicCofinal}). Others concern cofinal types of ultrafilters  (\cite{Milovich}, \cite{DobrinenTodorcevic}, \cite{DobrinenTodorcevicRamseyTukey1}, \cite{DobrinenTodorcevicRamseyTukey2}, \cite{RaghavanTodorcevicCofinal}, \cite{RaghavanShelah}, \cite{KuzeljevicRaghavan}). Additional research concerns the Tukey order on various sets in topological spaces (\cite{GartsideMamatelashviliCompact}, \cite{GartsideMamatelashviliOmega1}, \cite{GartsideMamatelashviliTukeyPCF}). Yet another batch of results studies the number of cofinal types of partial orders of various sizes below $\aleph_\om$ (\cite{KuzeljevicTodorcevic}, \cite{ShalevCofinal}). See also \cite{FremlinTukeyMeasure} and \cite{MooreSolecki}.

Cofinal structure has been studied by set theorists coming from a different angle.  Especially important for us is Shelah's theory of possible cofinalities, or PCF theory for short. The main objects of study in PCF theory are reduced products of sets of regular cardinals, modulo an ideal. Shelah has developed the theory in a series of papers which culminated in the book \cite{ShelahCardinalBook}. PCF theory has had dramatic implications for our understanding of cardinal arithmetic (see \cite{AbrahamMagidorHandbook}), as well as plenty of applications both inside and outside of set theory, such as \cite{KojmanShelahDowker}. See \cite{ShelahSkeptics} for a discussion of further applications.

These two ways of studying cofinal structure are related (and we will discuss this more later): given a set $A$ of regular cardinals, we consider the Tukey spectrum of $A$, which consists of all regular cardinals which are Tukey below $(\prod A,<)$ (i.e., $\prod A$ with the pointwise domination ordering). We denote this by $\spec(A)$. It follows quickly from the definitions (which we give later) that for any set $A$ of regular cardinals, $\pcf(A)\seq\spec(A)$. In this paper, we are concerned with the following general question: 

\begin{question}\label{q:theQ}
Does $\zfc$ prove that for any set $A$ of regular cardinals, $\pcf(A)=\spec(A)$?
\end{question}

To our knowledge, the only result, so far, which addresses this question is due to Gartside and Mamatelashvili (\cite{GartsideMamatelashviliTukeyPCF}) who have a proof showing that if $A$ is any \emph{progressive} set of regular cardinals, then $\pcf(A)=\spec(A)$ (``progressive" is a common assumption when doing PCF theory). However, there is a gap in their proof. We address this gap later, observing that their argument rather shows that if $A$ is progressive, then $\spec(A)\seq\pcf(A)\cup\lim(\pcf(A))$.\footnote{Since $\spec(A)$ consists, by definition, of regular cardinals, if $A$ is progressive and $\spec(A)\neq\pcf(A)$, then $\pcf(A)$ has a regular limit point.} Additional assumptions on $\pcf(A)$ then guarantee equality. However, the status of Question \ref{q:theQ} when $A$ is not progressive is far from clear. 

In this work, we prove various results related to Question \ref{q:theQ}. After a review of the basics of the Tukey order and PCF theory in Section \ref{sec:review}, we turn in Section \ref{sec:SpecBad?} to the question of how much bigger $\spec(A)$ can be than $\pcf(A)$. We review the theorem from \cite{GartsideMamatelashviliTukeyPCF} and address the gap in their proof. Then we turn to showing that Question \ref{q:theQ} has a consistent positive answer. We also discuss circumstances under which, for all $A$, $\spec(A)$ is no worse than $\pcf(A)\cup\lim(A)$. In Section \ref{sec:SpecSmallLCs} we address the role that small large cardinals (Mahlo and weakly compact) have in excluding a regular limit $\ka$ from $\spec(A)$ (where $A\seq\ka$) or for ensuring that $\ka\in\spec(A)$. The upshot of these results is that they may reduce the options for showing that Question \ref{q:theQ} has a consistent negative answer which is witnessed by a forcing separating $\spec(A)$ and $\pcf(A)$ (if such exists). In the last main section, Section \ref{sec:strongspec}, we show that a subset of the Tukey spectrum (what we call the ``strong part" of the Tukey spectrum) is sufficiently strong to be able to ``lift" the existence of J{\'o}nsson algebras; this generalizes Shelah's celebrated result \cite{ShelahJonsson} that scales in PCF theory can lift the existence of J{\'o}nsson algebras.\\

\textbf{Acknowledgements} We would like to thank Will Brian, James Cummings, Todd Eisworth, and Paul Gartside for many helpful conversations about the Tukey order and PCF theory and for suggesting ways of extending this line of research.

\section{A quick overview of Tukey and PCF}\label{sec:review}

In this section we review the basics of the Tukey order and PCF theory which are relevant for this paper. 

\begin{remark}
Throughout the paper, all posets are assumed to be directed.
\end{remark} 

First we recall the definition of the Tukey order (see \cite{GartsideMamatelashviliCompact} for a detailed development of these ideas).

\begin{definition}
A poset $Q$ is said to be a \emph{Tukey quotient} of $P$ if there exists a function $\vp:P\lra Q$ which preserves cofinal sets. We denote this by $P\geq_TQ$.
\end{definition}

$P\geq_TQ$ is equivalent to the existence of a map $\psi:Q\lra P$ which preserves unbounded sets.

\begin{definition}
Suppose that $\ka\geq\lam\geq\mu$ are cardinals. We say that a poset $P$ has \emph{calibre} $(\ka,\lam,\mu)$ if for all $\ka$-sized $P'\seq P$ there is a $\lam$-sized $R\seq P'$ so that every $\mu$-sized $B\seq R$ is bounded in $P$.

We say simply that $P$ has \emph{calibre} $\ka$ if $P$ has calibre $(\ka,\ka,\ka)$.
\end{definition}

Note that the definition of $P$ having calibre $\ka$ simplifies to the following: every $\ka$-sized $P'\seq P$ has a $\ka$-sized subset which is bounded in $P$.

The following item connects the ideas of $\geq_T$ and calibre.

\begin{proposition}\label{prop:TukeyCalibre}
Suppose that $\ka$ is regular, that $P$ has calibre $(\ka,\lam,\mu)$, and that $P\geq_TQ$. Then $Q$ has calibre $(\ka,\lam,\mu)$ too.
\end{proposition}

The next item is particularly relevant for us.

\begin{proposition}
For a regular cardinal $\ka$, a poset $P$ fails to have calibre $\ka$ iff $P\geq_T\ka$.
\end{proposition}

Now we define the Tukey spectrum of a poset.

\begin{definition}
The \emph{Tukey spectrum} of a poset $P$ is denoted by $\spec(P)$ and defined to be $\spec(P):=\lb\ka:P\geq_T\ka\we\ka\text{ is regular}\rb$.

When $A$ is a set of regular cardinals, we let $\spec(A)$ abbreviate $\spec\left(\prod A,<\right)$.
\end{definition}

\begin{remark}
$\spec(P)$ consists of all regular $\ka$ so that $P$ does \textbf{\emph{not}} have calibre $\ka$.
\end{remark}

It is helpful to get a better handle on what $\ka\in\spec(A)$ means in the specific case that $A$ is a set of regular cardinals. Indeed, $\ka\in\spec(A)$ iff there exists a set $\cal{F}$ of $\ka$-many functions in $\prod A$ so that every $\cal{F}_0\in[\cal{F}]^\ka$ is unbounded in $(\prod A,<)$, i.e., $A$ with the pointwise domination ordering. That is to say, there is \emph{at least one} coordinate $a\in A$ so that
$$
\lb f(a):f\in\cal{F}_0\rb
$$
is unbounded in the regular cardinal $a$. We give a name to these coordinates in the next definition.

\begin{definition}\label{def:ub}
Suppose that $A$ is a set of regular cardinals $\cal{F}\seq\prod A$. A cardinal $a\in A$ is called an \emph{unbounded coordinate} of $\cal{F}$ if $\lb f(a):f\in\cal{F}\rb$ is unbounded in $a$.

We let $\operatorname{ub}(\cal{F})$ denote the set of unbounded coordinates of $\cal{F}$.
\end{definition}

We address the question of how many coordinates are in $\operatorname{ub}(\cal{F})$ in Section \ref{sec:strongspec}.\\

The following lemma is a standard part of Tukey-ology.

\begin{lemma}\label{lemma:specproduct}
$\spec(P\times Q)=\spec(P)\cup\spec(Q)$. Hence if $A$ and $B$ are sets of regular cardinals, $\spec(A\cup B)=\spec(A)\cup\spec(B)$.
\end{lemma}

This captures all of the basics of the Tukey order that we need. We now review some of the central results in PCF theory, beginning with the central definition (see \cite{AbrahamMagidorHandbook} for a clear and detailed exposition of these and related ideas).

\begin{definition}
Let $A$ be a set of regular cardinals. 
$$
\pcf(A):=\lb\cf\left(\prod A/D\right):D\text{ is an ultrafilter on }A\rb.
$$
\end{definition}

The following are routine facts about the pcf function.

\begin{fact}\label{fact:babyPCF}
Suppose that $A$ and $B$ are sets of regular cardinals.
\begin{enumerate}
    \item $A\seq\pcf(A)$ (by using principal ultrafilters).
    \item If $A\seq B$, then $\pcf(A)\seq\pcf(B)$ (since any ultrafilter on $A$ can be extended to one on $B$).
    \item $\pcf(A\cup B)=\pcf(A)\cup\pcf(B)$ (since any ultrafilter on $A\cup B$ contains either $A$ or $B$).
\end{enumerate}
\end{fact}

The next lemma follows almost immediately from the definitions:

\begin{lemma}\label{lemma:onedirection}
For any set $A$ of regular cardinals, $\pcf(A)\seq\spec(A)$.
\end{lemma}

A very useful assumption when doing PCF theory is the following:

\begin{definition}
A set $A$ of regular cardinals is \emph{progressive} if $|A|<\min(A)$.
\end{definition}

We next define certain ideals which are naturally associated to the cardinals in $\pcf(A)$. $A$ being progressive plays an important role in the development of the ideas that follow.

\begin{definition}
Let $A$ be a set of regular cardinals and $\lam$ a cardinal (singular or regular). Define the ideal $J_{<\lam}[A]$ to consist of all $B\seq A$ so that for any ultrafilter $D$ on $A$ with $B\in D$, $\cf(\prod A/D)<\lam$.
\end{definition}

If the set $A$ is clear from context, we write $J_{<\lam}$ instead of $J_{<\lam}[A]$. A crucial fact about the ideal $J_{<\lam}$ is the following:

\begin{proposition}\label{prop:directed}
Suppose that $A$ is progressive. Then for any cardinal $\lam$, $\prod A/J_{<\lam}$ is $<\lam$-directed (that is, any set of fewer than $\lam$-many functions in $\prod A$ has an upper bound mod $J_{<\lam}$).
\end{proposition}

A major theorem in PCF theory is the existence of generators, which we define now.

\begin{definition}
Let $A$ be a set of regular cardinals and $\lam\in\pcf(A)$. A \emph{generator} for $\lam$ is a set $B_\lam\seq A$ in $J_{<\lam^+}$ so that $J_{<\lam}\cup\lb B_\lam\rb$ generates $J_{<\lam^+}$. In other words, given $X\seq A$, $X\in J_{<\lam^+}$ iff $X\bsl B_\lam\in J_{<\lam}$.
\end{definition}

Thus a generating set $B_\lam$ is a maximal set in $J_{<\lam^+}$, modulo $J_{<\lam}$; they are unique modulo $J_{<\lam}$. The following is often proven using universal cofinal sequences:

\begin{proposition}\label{prop:generator}
Suppose that $A$ is progressive. Then for any $\lam\in\pcf(A)$, there is a generator for $\lam$.
\end{proposition}

Generators give an instance of compactness:

\begin{proposition}\label{prop:compactness}
Suppose that $A$ is progressive and $B\seq A$. Let $\la B_\lam:\lam\in\pcf(A)\ra$ be a sequence of generators. Then there exists a finite decreasing sequence $\lam_0>\dots>\lam_n$ of elements of $\pcf(A)$ with $\lam_0=\max\operatorname{pcf}(B)$ so that
$$
B\seq\bigcup_{i\leq n}B_{\lam_i}.
$$
\end{proposition}

Another application of generators and related ideas is the following fact:

\begin{proposition}\label{prop:maxpcf}
Suppose that $A$ is progressive. Then $\pcf(A)$ has a maximum element, and moreover,
$$
\max\operatorname{pcf}(A)=\cf\left(\prod A,<\right).
$$
\end{proposition}
\vspace{.08in}

The last collection of background facts to review concerns weak compactness.

\begin{definition}\label{def:alphaModel}
Let $\al$ be an inaccessible cardinal. We say that a transitive set $M$ is an \emph{$\al$-model} if $M\models\mathsf{ZFC}^-$, $M$ has size $\al$, $M$ is closed under $<\al$-sequences, and $\al\in M$.
\end{definition}

Now we recall the following characterization of weak compactness (see \cite{Hauser}):

\begin{fact}
An inaccessible cardinal $\ka$ is weakly compact if and only if for any $\ka$-model $M$, there exist a $\ka$-model $N$ and an elementary embedding $j:M\to N$ with $\operatorname{crit}(j)=\ka$.
\end{fact}

In the context of the previous fact, note that $j:M\to N$ gives rise to an $M$-normal ultrafilter $\cal{U}_j:=\lb X\in M:X\seq\ka\we\ka\in j(X)\rb$ on $\cal{P}(\ka)\cap M$ (recall that an $M$-ultrafilter $\cal{U}$ is $M$-normal if for any $A\in\cal{U}$ and any regressive $f:A\to\ka$ with $f\in M$, $f$ is constant on a set in $\cal{U}$). Moreover, since  $M$ is closed under $<\ka$-sequences (by definition of a $\ka$-model), this ultrafilter is $<\ka$-closed in $V$ (but of course it only measures subsets of $\ka$ in $M$).

\section{How bad can Spec be?}\label{sec:SpecBad?}

In this section, we study various conditions which guarantee that either (a) $\spec(A)$ is no worse than $\pcf(A)$ together with (regular) limit points of $\pcf(A)$ or (b) $\spec(A)$ is no worse than $\pcf(A)$ together with regular limit points of $A$. These results, when coupled with anti large cardinal hypotheses, give a consistent positive answer to Question \ref{q:theQ}.

We begin by looking at the theorem from \cite{GartsideMamatelashviliTukeyPCF} which we mentioned in the introduction. Addressing a small gap in their argument, what their proof shows is that for any progressive $A$, $\spec(A)$ can at worst add regular limits of $\pcf(A)$. Additional assumptions on $\pcf(A)$ then give the equality of $\pcf(A)$ and $\spec(A)$. We first have some notation.


\begin{notation}\label{notation:max}
Given functions $g_0,\dots,g_n$ in some product $\prod A$ of regular cardinals, we let $\max(g_0,\dots,g_n)$ be the function $h\in\prod A$ defined by 
$$
h(a):=\max\lb g_0(a),\dots,g_n(a)\rb.
$$
\end{notation}
\vspace{.05in}

\begin{theorem}\label{thm:GM} (Almost entirely \cite{GartsideMamatelashviliTukeyPCF})
Suppose that $A$ is a progressive set of regular cardinals. Then $\spec(A)\seq\pcf(A)\cup\lim(\pcf(A))$.\footnote{The original statement of their theorem applies to $A$ which are finite unions of progressive sets and uses what we are calling Lemma \ref{lemma:specproduct} to reduce to the case of a single progressive set.}

In particular, if $\pcf(A)$ is closed under regular limits, or if $\pcf(A)$ has no regular limits (for example, if $\pcf(A)$ is itself progressive), then $\pcf(A)=\spec(A)$.
\end{theorem}
\begin{proof}
We begin with some set-up. Since $A$ is progressive, we may apply Proposition \ref{prop:generator} to fix a sequence $\la B_\lam:\lam\in\pcf(A)\ra$ of generators for $\pcf(A)$. For each $\lam\in\pcf(A)$, $B_\lam$ is progressive, being a subset of the progressive set $A$. Thus $\lam=\operatorname{maxpcf}(B_\lam)=\cf\left(\prod B_\lam,<\right)$, using $B_\lam\in J_{<\lam^+}\bsl J_{<\lam}$ for the first equality and applying Proposition \ref{prop:maxpcf} for the second equality. Therefore, for each $\lam\in\pcf(A)$, we may choose a sequence $\vec{f}^\lam=\la f^\lam_\al:\al<\lam\ra$ of functions in $\prod A$ so that $\la f^\lam_\al\res B_\lam:\al<\lam\ra$ is cofinal in $\left(\prod B_\lam,<\right)$.

Fix a cardinal $\ka\in\spec(A)$, suppose that $\ka$ is not a limit point of $\pcf(A)$, and we will show that $\ka\in\pcf(A)$. Because $\ka$ is not a limit point of $\pcf(A)$, we know that $|\pcf(A)\cap\ka|<\ka$; this will permit us to apply a pigeonhole argument at a crucial step later in the proof.

Since $\ka\in\spec(A)$, let $\la f_\ga:\ga<\ka\ra$ enumerate a set $\cal{F}$ of $\ka$-many functions in $\prod A$ so that every $\ka$-sized subset of $\cal{F}$ is unbounded in $\left(\prod A,<\right)$. To show that $\ka\in\pcf(A)$, it suffices to show that $J_{<\ka}\neq J_{<\ka^+}$. Since $\cal{F}$ is bounded in $\prod A/J_{<\ka^+}$ (because this reduced product is $<\ka^+$-directed, by Proposition \ref{prop:directed}), it in turn suffices to show that $\cal{F}$ is unbounded in $\prod A/J_{<\ka}$.

We work by contradiction and suppose that $\cal{F}$ is bounded in $\prod A/J_{<\ka}$. Our goal is to define a set $\cal{X}$ of fewer than $\ka$-many functions in $\prod A$ so that every $f\in\cal{F}$ is pointwise below some $g\in\cal{X}$. Supposing we can define such an $\cal{X}$, we obtain a contradiction as follows: since $\ka$ is regular, there is a single $g\in\cal{X}$ which is pointwise above $\ka$-many elements of $\cal{F}$, and this contradicts the fact that $\cal{F}$ witnesses that $\ka\in\spec(A)$.

We begin the construction of $\cal{X}$ as follows. Since $\cal{F}$ is bounded in $\prod A/J_{<\ka}$, let $g$ be a bound. For each $n\in\om$, each sequence $\vec{\lam}=\lam_0>\dots>\lam_n$ in $\pcf(A)\cap\ka$, and each sequence $\vec{\al}=\la\al_0,\dots,\al_n\ra$ with $\al_i\in\lam_i$ for all $i\leq n$, we let $h(\vec{\al},\vec{\lam})$ be the function $\max(f^{\lam_0}_{\al_0},\dots,f^{\lam_n}_{\al_n},g)$; see Notation \ref{notation:max}. We let $\cal{X}$ be the set of such $h$. Since by assumption $\pcf(A)\cap\ka$ is bounded below $\ka$, we see that $\cal{X}$ consists of fewer than $\ka$-many functions.

Now we verify that $\cal{X}$ has the desired property. Recalling that $\la f_\al:\al<\ka\ra$ enumerates $\cal{F}$, fix some $\al<\ka$. Since $f_\al<_{J_{<\ka}}g$, we know that 
$$
z_\al:=\lb a\in A:f_\al(a)\geq g(a)\rb
$$ 
is in $J_{<\ka}$. Thus $\max\pcf(z_\al)<\ka$, and so by Proposition \ref{prop:compactness}, we may find a sequence $\lam_0>\dots>\lam_n$ in $\pcf(A)\cap\ka$ so that $z_\al\seq\bigcup_{i\leq n}B_{\lam_i}$. Then, for each $i\leq n$, pick $\al_i<\lam_i$ so that
$$
f_\al\res B_{\lam_i}<f^{\lam_i}_{\al_i}\res B_{\lam_i}
$$
(i.e., this holds pointwise on $B_{\lam_i}$).
Then we see that
$$
f_\al<\max(f^{\lam_0}_{\al_0},\dots,f^{\lam_n}_{\al_n},g).
$$
Indeed, if $a\notin z_\al$, then $f_\al(a)<g(a)$. On the other hand, if $a\in z_\al$, there is some $i\leq n$ so that $a\in B_{\lam_i}$, and hence $f_\al(a)<f^{\lam_i}_{\al_i}(a)$. Since the function $\max(f^{\lam_0}_{\al_0},\dots,f^{\lam_n}_{\al_n},g)$ is in $\cal{X}$, this completes the main part of the proof.

For the ``in particular" part of the theorem, observe that if $\pcf(A)$ is progressive, then there are no regular limit $\ka$ in $\lim(\pcf(A))$. Otherwise, $|\pcf(A)|\geq\ka >\min(\pcf(A))$, contradicting that $\pcf(A)$ is progressive.
\end{proof}

In the rest of this section, we work to isolate a condition (Theorem \ref{thm:limpcf}) which is consistent with $\zfc$ and which implies that $\spec(A)\seq \pcf(A)\cup\lim(\pcf(A))$ for all $A$. In fact, it implies the stronger result that $\spec(A)\seq \pcf(A)\cup\lim(A)$ for all $A$. We have some preliminary lemmas. The first of these illustrates a ``dropping" phenomenon in the Tukey spectrum; we will use this lemma as part of a later inductive argument.

\begin{lemma}\label{lemma:specdrop}
Suppose that $A$ is a set of regular cardinals. Let $\ka\in\spec(A)\cap\sup(A)$. Then $\ka\in\spec(A\cap(\ka+1))$.
\end{lemma}
\begin{proof} Set $\mu:=\sup(A)$; note that we are making no assumption about whether or not $\mu\in A$. Since $\ka\in\spec(A)$, let $\cal{F}$ be a set of $\ka$-many functions in $\prod A$ so that for all $\cal{F}_0\in[\cal{F}]^\ka$, $\cal{F}_0$ is unbounded in $(\prod A,<)$. Let 
$$
\cal{F}_{\leq\ka}:=\lb f\res (\ka+1):f\in\cal{F}\rb
$$ 
and 
$$
\cal{F}_{>\ka}:=\lb f\res (\ka+1,\mu]:f\in\cal{F}\rb.
$$
Note that since $\ka<\mu=\sup(A)$, $A\bsl (\ka+1)$ is a non-empty set of regular cardinals. Also, note that $\cal{F}_{>\ka}$ is bounded in $\prod (A\bsl (\ka+1),<)$, since $\cal{F}_{>\ka}$ consists of at most $\ka$-many functions, and we can take a sup of $\ka$-many elements on each coordinate in $A\bsl(\ka+1)$. 

We next argue that $\cal{F}_{\leq\ka}$ has size $\ka$ and that every $\ka$-sized subset is unbounded in $(\prod A\cap(\ka+1),<)$; this will show the desired result. First suppose for a contradiction that $\cal{F}_{\leq\ka}$ has size $<\ka$. Then there is a single $\bar{f}\in\prod A\cap(\ka+1)$ so that for $\ka$-many $g\in\cal{F}$, $g\res(\ka+1)=\bar{f}$. Let $\cal{F}_0$ be this set of $g\in\cal{F}$. But then $\cal{F}_0$ is bounded in the entire product $(\prod A,<)$, using the observation from the previous paragraph to bound the elements of $\cal{F}_0$ on coordinates in $A$ above $\ka$. A nearly identical argument shows that $\cal{F}_{\leq\ka}$ must satisfy that every $\ka$-sized subset is unbounded in $(\prod A\cap(\ka+1),<)$.
\end{proof}

The next lemma will also be used in the proof of Theorem \ref{thm:limpcf}:

\begin{lemma}\label{lemma:notMahlo}
Suppose that $\ka$ is a regular limit cardinal. Let $A\seq\ka$ be a non-stationary set of regular cardinals which is unbounded in $\ka$, and let $D$ be an ultrafilter on $A$ which extends the tail filter on $A$. Then $\cf(\prod A/D)\geq\ka^+$.
\end{lemma}
\begin{proof}
To begin, let $\la\mu_\nu:\nu<\ka\ra$ be the increasing enumeration of $A$, and let $\la\zeta_i:i<\ka\ra$ enumerate a club $C\seq\ka$ with $C\cap A=\es$. By relabeling if necessary, we take $\zeta_0=0$.

Suppose for a contradiction that $\cf(\prod A/D)\leq\ka$; then the cofinality must equal exactly $\ka$ since $D$ extends the tail filter on $A$. Let $\vec{f}=\la f_\al:\al<\ka\ra$ be a sequence of functions in $\prod A$ which is increasing and cofinal modulo $D$. We obtain our contradiction by showing that $\vec{f}$ is bounded in $\prod A$, modulo the tail filter on $A$, and hence modulo $D$.

Given $\nu<\ka$ (corresponding to $\mu_\nu$), let $i<\ka$ so that $\nu\in[\zeta_i,\zeta_{i+1})$ (such an $i$ exists since $C$ is club and $\zeta_0=0$). The following observation, though simple, is crucial: for each $i<\ka$, $\mu_{\zeta_i}>\zeta_i$: indeed, $\mu_{\zeta_i}\geq\zeta_i$ since the enumeration of $A$ is increasing. But $\zeta_i\in C$ and $\mu_{\zeta_i}\in A\seq\ka\bsl C$, so $\mu_{\zeta_i}>\zeta_i$.  Therefore $\mu_\nu$ is also strictly above $\zeta_i$, and so
$$
\sup\lb f_\al(\mu_\nu):\al\leq\zeta_i\rb
$$
is below $\mu_\nu$. We let $g(\mu_\nu)<\mu_\nu$ be above this sup. $g$ is then a member of $\prod A$.

We finish the proof by showing that $g$ bounds each of the $f_\xi$ on a tail. To this end, fix $\al<\ka$, and let $i<\ka$ so that $\al\leq\zeta_i$. We claim that for all $\nu\geq\zeta_i$, $g(\mu_\nu)>f_\al(\mu_\nu)$. Fix $\nu\geq\zeta_i$, and let $j\geq i$ so that $\nu\in[\zeta_j,\zeta_{j+1})$. Then
$$
g(\mu_\nu)>\sup\lb f_\be(\mu_\nu):\be\leq\zeta_j\rb\geq f_\al(\mu_\nu),
$$
where the last inequality follows since $\al\leq\zeta_i\leq\zeta_j$.
\end{proof}

\begin{theorem}\label{thm:limpcf}
Let $V$ be a model of $\zfc$ in which $2^\mu=\mu^+$ for all limit cardinals $\mu$ and in which there are no Mahlo cardinals. Then $V$ satisfies that for any set $A$ of regular cardinals, $\spec(A)\seq\pcf(A)\cup\lim(A)$.

Consequently, if $2^\mu=\mu^+$ for all singular $\mu$ and if there are no regular limit cardinals, then for all $A$, $\spec(A)=\pcf(A)$.
\end{theorem}
\begin{proof}
Fix such a $V$ (for example, work in $L$ up to the first $\lam$ which is Mahlo in $L$, if such a $\lam$ exists, and otherwise work in all of $L$). Let $A$ be a set of regular cardinals, and we will prove the result by induction on the order type of $A$.

We first dispense with the case when $\max(A)$ exists. Let $\lam:=\max(A)$. Then, applying Lemma \ref{lemma:specproduct} and Fact \ref{fact:babyPCF}, as well as our inductive assumption, we get
\begin{align*}
    \spec(A)&=\spec((A\cap\lam)\cup\lb\lam\rb)\\
    &=\spec(A\cap\lam)\cup\underbrace{\spec(\lb\lam\rb)}_{=\lb\lam\rb}\\
    &\seq \pcf(A\cap\lam)\cup\lim(A\cap\lam) \cup\lb\lam\rb\\
    &=\pcf(A\cap\lam)\cup\pcf(\lb\lam\rb)\cup \lim(A\cap\lam)\\
    &=\pcf(A)\cup\lim(A\cap\lam)\\
    &\seq\pcf(A)\cup\lim(A).
\end{align*}

Now we suppose that $A$ does not have a max. Let $\mu:=\sup(A)$, and fix $\ka\in\spec(A)$. We have a few cases on $\ka$.

First suppose that $\ka<\mu$. Then by Lemma \ref{lemma:specdrop}, $\ka\in\spec(A\cap(\ka+1))$. Since the order type of $\bar{A}:=A\cap(\ka+1)$ is less than the order type of $A$, we apply the induction hypothesis to conclude that
$$
\ka\in\spec(\bar{A})\seq\pcf(\bar{A})\cup\lim(\bar{A})\seq\pcf(A)\cup\lim(A).
$$

Now suppose that $\ka\geq\mu$. If $\ka=\mu$ (and in particular, $\mu$ is a regular limit cardinal), then because $A$ is unbounded in $\mu$,  $\ka=\mu\in\lim(A)$.

The final case is that $\ka>\mu$ (here $\mu$ may be either regular or singular). Since $\mu$ is a limit cardinal, our cardinal arithmetic assumption implies that $\prod A$ has size $\mu^\mu=\mu^+$. Hence no cardinal greater than $\mu^+$ is in $\spec(A)$ or in $\pcf(A)$.

To finish the proof in this final case, we will argue that $\mu^+\in\pcf(A)$. Towards this end, let $D$ be an ultrafilter on $A$ which extends the tail filter on $A$. Then $\cf(\prod A/D)\geq\mu=\sup(A)$. If $\mu$ is singular, then the regular cardinal $\cf(\prod A/D)$ is greater than $\mu$. On the other hand, if $\mu$ is regular, then $\mu$ is not a Mahlo cardinal by assumption. This in turn, with the help of Lemma \ref{lemma:notMahlo}, implies that $\cf(\prod A/D)>\mu$. Thus in either case on $\mu$, $\cf(\prod A/D)>\mu$. This cofinality must be exactly $\mu^+$, however, since $\prod A$ has size $\mu^+$.

For the ``consequently" part of the theorem, recall that $\spec(A)$ consists, by definition, of regular cardinals. Thus if there are no regular limit cardinals, $\lim(A)$ contains no regular cardinals, and this implies that $\spec(A)\seq\pcf(A)$. By Lemma \ref{lemma:onedirection}, we conclude that $\pcf(A)=\spec(A)$.
\end{proof}

Thus Question \ref{q:theQ} has a consistent positive answer.

\section{Small Large Cardinals and the Tukey Spectrum}\label{sec:SpecSmallLCs}

In this section, we prove some results showing the relationship between certain small large cardinals (Mahlo and weakly compact) and the Tukey spectrum. The first of this gives a sufficient condition for including a regular limit cardinal in the Tukey spectrum. After this, we prove Theorem \ref{thm:notspec} which gives a sufficient condition for excluding a cardinal from $\spec(A)$. After the proof of Theorem \ref{thm:notspec}, we comment on applications.

\begin{proposition}\label{prop:InSpec}
Suppose that $\ka$ is a Mahlo cardinal and that $A\seq\ka$ is any stationary set of regular cardinals. Then $\ka\in\spec(A)$.
\end{proposition}
\begin{proof}
We show that functions which are constant on a tail witness the result. For each $\al<\ka$, let $f_\al$ be the function in $\prod A$ which takes value $0$ on all $a\in A$ with $a\leq\al$, and which takes value $\al$ on all $a\in A$ with $\al<a$. We claim that the sequence $\la f_\al:\al<\ka\ra$ enumerates a set which witnesses that $\ka\in\spec(A)$.

Towards this end, let $X\in[\ka]^\ka$, and we will show that $\la f_\al:\al\in X\ra$ is unbounded in $(\prod A,<)$. Since $A$ is stationary, we may find some $a\in A\cap\lim(X)$. Then for all $\al\in X\cap a$, $f_\al(a)=\al$. Since $a$ is a limit point of $X$, we have that $\lb f_\al(a):\al\in X\cap a\rb$ is cofinal in $a$. This completes the proof.
\end{proof}

The next result shows that we can use weak compactness to exclude a regular limit $\ka$ from $\spec(A)$, for certain $A$.

\begin{theorem}\label{thm:notspec}
Suppose that $\ka$ is weakly compact and that $A\seq\ka$ is a non-stationary set of regular cardinals which is unbounded in $\ka$. Then $\ka\notin\spec(A)$.

Therefore, if $\ka$ is weakly compact and $A\seq\ka$ is an unbounded set of regular cardinals, $\ka\in\spec(A)$ iff $A$ is stationary.
\end{theorem}
\begin{proof}
Let $C\seq\ka$ be a club with $C\cap A=\es$. Enumerate $C$ in increasing order as $\la\zeta_i:i<\ka\ra$, where we assume $\zeta_0=0$. Also, enumerate $A$ in increasing order as $\la\mu_i:i<\ka\ra$. As in the proof of Lemma \ref{lemma:notMahlo}, we have that for each $i<\ka$, $\mu_{\zeta_i}>\zeta_i$.

Next, let $\la f_\al:\al<\ka\ra$ be an enumeration of a set $\cal{F}$ of $\ka$-many functions in $\prod A$, and we will show that $\cal{F}$ does not witness that $\ka\in\spec(A)$. Fix a $\ka$-model $M$ (see Definition \ref{def:alphaModel}) which contains $A$, $C$, and $\la f_\al:\al<\ka\ra$ as elements. By the weak compactness of $\ka$, let $\U$ be an $M$-normal ultrafilter on $\cal{P}(\ka)\cap M$.

Our strategy is to use $\cal{U}$ to successively freeze out longer and longer initial segments of many functions on the sequence $\la f_\al:\al<\ka\ra$. We will then bound their tails using the non-stationarity of $A$.

For each $i<\ka$, the product
$$
\prod_{\nu\in[\zeta_i,\zeta_{i+1})}\mu_\nu
$$
is a member of $M$, and hence a subset of $M$. Since $\ka$ is strongly inaccessible, this product has size $<\ka$. Applying the fact that $\cal{U}$ is a $\ka$-complete ultrafilter on $\cal{P}(\ka)\cap M$, we may find a function $\vp_i$ in $\prod_{\nu\in[\zeta_i,\zeta_{i+1})}\mu_\nu$
so that
$$
Z_i:=\lb\be<\ka:f_\be\res\lb\mu_\nu:\nu\in[\zeta_i,\zeta_{i+1})\rb=\vp_i\rb\in\U.
$$
Next, define an increasing sequence $\la\be_j:j<\ka\ra$ below $\ka$ so that for all $j<\ka$,
$$
\be_j\in\bigcap_{i<j}Z_i.
$$
This also uses the completeness of $\cal{U}$ to see that for each $j<\ka$, $\bigcap_{i<j}Z_i$ is in $\cal{U}$ and has size $\ka$.

As a result of this freezing out, we have the following: for a fixed $i<\ka$, and all $j>i$, $\be_j\in Z_i$. Hence, for all $\nu\in[\zeta_i,\zeta_{i+1})$,
$$
f_{\be_j}(\mu_\nu)=\vp_i(\mu_\nu).
$$

Now consider an $i<\ka$ and $\nu\in[\zeta_i,\zeta_{i+1})$. Let
$$
R(\nu):=\lb f_{\be_j}(\mu_\nu):j<\ka\rb,
$$
i.e., all values of all of the $f_{\be_j}$ on the column $\mu_\nu\in A$. By applying the argument in the previous paragraph, we conclude that $R(\nu)$ in fact equals
$$\lb f_{\be_j}(\mu_\nu):j\leq i\rb\cup \lb f_{\be_j}(\mu_\nu):j> i\rb
=\lb f_{\be_j}(\mu_\nu):j\leq i\rb\cup \lb\vp_i(\mu_\nu)\rb.
$$

Now fix an arbitrary $\nu<\ka$. Since $C$ is a club and $\zeta_0=0$, there is an $i$ so that $\nu\in [\zeta_i,\zeta_{i+1})$. We claim that $R(\nu)=\lb f_{\be_j}(\mu_\nu):j\leq i\rb\cup \lb\vp_i(\mu_\nu)\rb$ has size less than $\mu_\nu$. If $i$ is finite, then $R(\nu)$ is finite, and hence has size smaller than $\mu_\nu$ (which is, after all, an infinite cardinal). On the other hand, if $i$ is infinite, then
$$
|R(\nu)|\leq |i|\leq i\leq \zeta_i<\mu_{\zeta_i}\leq\mu_\nu.
$$

Consequently, for each $\nu<\ka$, $R(\nu)$ is bounded in $\mu_\nu$, as $\mu_\nu$ is regular. But by definition of $R(\nu)$, this means that for all $\nu<\ka$, $\lb f_{\be_j}(\mu_\nu):j<\ka\rb$ is bounded in $\mu_\nu$. Thus $\lb f_{\be_j}:j<\ka\rb$ enumerates a set of $\ka$-many functions from $\cal{F}$ which is bounded in $(\prod A,<)$. Since $\cal{F}$ was arbitrary, this shows that $\ka\notin\spec(A)$.\\

For the final statement of the theorem, note that if $A\seq\ka$ is non-stationary, then since $\ka$ is Mahlo, Proposition \ref{prop:InSpec} shows that $\ka\notin\spec(A)$.
\end{proof}

Note that the converse of the above theorem may fail, since a regular limit cardinal which is not weakly compact may also fail to be in $\spec(A)$:

\begin{corollary}\label{cor:notspec}
Suppose that $\ka$ is weakly compact and that $A\seq\ka$ is a non-stationary set of regular cardinals unbounded in $\ka$. Let $\ps:=\operatorname{Add}(\om,\ka)$, the poset to add $\ka$-many Cohen subsets of $\om$. Then $\ps$ forces that $\ka\notin\spec(A)$.
\end{corollary}
\begin{proof}
Let $\la\dot{f}_\al:\al<\ka\ra$ be a sequence of $\ps$-names for elements of $\prod A$. We will find an $X\in[\ka]^\ka$ in $V$ so that $\ps$ forces that $\la\dot{f}_\al:\al\in X\ra$ is bounded.

Indeed, using the c.c.c. of $\ps$, for each $\al<\ka$, we may find a function $\vp_\al\in\prod A$ so that $\ps\Vdash(\forall a\in A)\,[\dot{f}_\al(a)<\vp_\al(a)]$. Namely, let $\vp_\al(a)$ be above the sup of the countably-many $\ga\in a$ so that $\ga$ is forced to be the value of $\dot{f}_\al(a)$ by some condition in $\ps$.

By the Theorem \ref{thm:notspec}, let $X\in[\ka]^\ka$ so that $\la\vp_\al:\al\in X\ra$ is bounded in the product $(\prod A,<)$, say with $h$ as a bound. Then $\ps$ forces that for each $\al\in X$, $\dot{f}_\al$ is pointwise below $h$.
\end{proof}

We conclude this section with a discussion of a promising suggestion of James Cummings about separating $\pcf(A)$ and $\spec(A)$. Given that $\pcf(A)\seq\spec(A)$ always holds, we'd like to create a forcing extension in which, for some $A$, there is a cardinal $\ka\in\spec(A)\bsl\pcf(A)$.

The strategy is to start with a cardinal $\ka$ which is at least Mahlo. Then let $A=\lb\mu^+:\mu<\ka\rb$ and attempt to force the existence of a set $\cal{F}$ of $\ka$-many functions in $\prod A$ which witnesses that $\ka\in\spec(A)$. This strategy appears promising due to the next observation.

\begin{lemma}\label{prop:notpcf} Suppose that $\ka$ is strongly inaccessible and that $A\seq\ka$ is a nonstationary set of regular cardinals unbounded in $\ka$. Then $\ka\notin\pcf(A)$.
\end{lemma}
\begin{proof} If $D$ is an ultrafilter on $A$ that concentrates on a bounded subset of $A$, then the strong inaccessibility of $\ka$ implies that $\cf(\prod A/D)<\ka$. On the other hand, if $D$ extends the tail filter on $\ka$, then by Lemma \ref{lemma:notMahlo}, $\cf(\prod A/D)>\ka$.
\end{proof}

While the above strategy is natural, problems remain. First, natural Easton-style forcings to add a witness to $\ka\in\spec(A)$ seem either to fail to add such a witness, or seem to change the Mahlo $\ka$ into a weakly, non-strongly inaccessible cardinal, i.e., they increase the continuum function below $\ka$ to take values at or above $\ka$. Thus the crucial assumption of Lemma \ref{prop:notpcf} fails. Or phrased differently, ultrafilters which concentrate on bounded subsets may give rise to reduced products with very high cofinality.

Moreover, Theorem \ref{thm:notspec} provides another obstacle: if a forcing $\ps$ places $\ka$ inside $\spec(A)\bsl\pcf(A)$ for some $A$ which is non-stationary and unbounded in $\ka$, then one of two things needs to happen. Either $\ka$ starts off as non-weakly compact (and this assumption plays a role in the argument) or $\ps$ must ensure that $\ka$ loses its weak compactness.

\section{The Strong Part of the Tukey Spectrum}\label{sec:strongspec}

In this section we introduce the idea of the strong part of the Tukey spectrum, and then we will show how this idea can be used in place of scales to lift the property of not being a J{\'o}nsson cardinal. Recall the notation $\operatorname{ub}(\cal{F})$ from Definition \ref{def:ub}.

First we make a simple observation about having infinitely-many unbounded coordinates.

\begin{lemma}
Let $A$ be a set of regular cardinals without a max, and let $\ka\in\spec(A)$ with $\ka\geq\sup(A)$. Let $\cal{F}$ be any witness that $\ka\in\spec(A)$. Then for all $\cal{F}_0\in[\cal{F}]^\ka$, $\operatorname{ub}(\cal{F}_0)$ is infinite.
\end{lemma}
\begin{proof}
Suppose otherwise, with $\cal{F}_0$ as a counterexample. Then since $\operatorname{ub}(\cal{F}_0)$ is finite and has a max below $\ka$, $\prod\operatorname{ub}(\cal{F}_0)$ has size below $\ka$. Let $\cal{F}_1\in[\cal{F}_0]^\ka$ so that the function $f\in\cal{F}_1\mapsto f\res\prod\operatorname{ub}(\cal{F}_0)$ is constant, say with value $\bar{f}$. Then we can bound all of $\cal{F}_1$ in the entire product $\prod A$ using $\bar{f}$ on the coordinates in $\operatorname{ub}(\cal{F}_0)\supseteq\operatorname{ub}(\cal{F}_1)$.
\end{proof}

Of course, $\operatorname{ub}(\cal{F})$ can very well be finite, or even a singleton, for instance, if $\ka$ is a member of $A$.

We want to isolate cases in which there are plenty of unbounded coordinates. This leads to the next definition.

\begin{definition}\label{def:strongpart}
Suppose that $A$ is a set of regular cardinals. The \emph{strong part} of the Tukey spectrum of $A$, denoted $\spec^*(A)$, consists of all regular $\lam$ satisfying the following: there is a set $\cal{F}\seq\prod A$ of size $\lam$, so that for every $\cal{F}_0\in[\cal{F}]^{\lam}$, $\operatorname{ub}(\cal{F}_0)$ is unbounded in $\sup(A)$.
\end{definition}

Thus $\lam\in\spec^*(A)$ iff there is a witness $\cal{F}$ to $\lam\in\spec(A)$ with the additional property that every $\lam$-sized subset has unboundedly-many unbounded coordinates.

Observe that if $A$ is a set of regular cardinals without a max, then $\spec^*(A)\cap\sup(A)=\es$. Indeed, if $\lam<\sup(A)$ and $\cal{F}\seq\prod A$ has size $\lam$, then we can bound $\cal{F}$ on all coordinates in $A\bsl\lam^+$.

Under cardinal arithmetic assumptions, it is easy to see that every $\lam\in\spec(A)\bsl\sup(A)$ is in $\spec^*(A)$:

\begin{lemma}
Suppose that $A$ is a set of regular cardinals with no max and that $\sup(A)$ is a strong limit cardinal (regular or singular). Then $\spec(A)\bsl\sup(A)\seq\spec^*(A)$.
\end{lemma}
\begin{proof}
Fix $\lam\in\spec(A)$ at least as large as $\sup(A)$. It suffices to show that if $\cal{F}$ is any witness that $\lam\in\spec(A)$, then $\operatorname{ub}(\cal{F})$ is unbounded. Fix $\de\in A$. Then the product $\prod (A\cap(\de+1))$ has size below $\sup(A)$. Thus there is $\cal{F}_0\in[\cal{F}]^\lam$ so that the function taking $f\in\cal{F}_0$ to $f\res (A\cap(\de+1))$ is constant on $\cal{F}_0$. Since $\cal{F}$ witnesses that $\lam\in\spec(A)$, we must have that $\operatorname{ub}(\cal{F}_0)$ is non-empty. Let $\de^*$ be the least element of $\operatorname{ub}(\cal{F}_0)$, and note that $\de^*>\de$, since we froze out the values of the functions in $\cal{F}_0$ on $A\cap(\de+1)$.
\end{proof}

We'd now like to connect the strong part of the Tukey spectrum with PCF theory. First we recall a few more definitions from PCF theory, beginning with the following standard version of the notion of a scale:

\begin{definition}
Let $\mu$ be a singular cardinal, $\la\mu_i:i<\cf(\mu)\ra$ an increasing sequence of regular cardinals which is cofinal in $\mu$, and $I$ an ideal on $\lb\mu_i:i<\cf(\mu)\rb$. Let $\vec{f}=\la f_\nu:\nu<\rho\ra$ be a sequence of functions in $\prod_{i<\cf(\mu)}\mu_i$. The tuple $(\vec{\mu},\vec{f},I)$ is called a \emph{scale of length $\rho$ modulo $I$} if $\vec{f}$ is increasing and cofinal in $\prod_{i<\cf(\mu)}\mu_i$ modulo $I$.

If $I$ is just the ideal of bounded subsets of $\lb\mu_i:i<\cf(\mu)\rb$, then we simply say that $(\vec{\mu},\vec{f})$ is a \emph{scale of length $\rho$} for $\mu$.
\end{definition}

We now connect scales with the strong part of the Tukey spectrum:

\begin{lemma}
Suppose that $(\vec{\mu},\vec{f})$ is a scale of length $\rho$, where $\rho$ is a regular cardinal. Then $\rho$ is in $\spec^*(A)$, where $A=\lb\mu_i:i<\cf(\mu)\rb$.
\end{lemma}
\begin{proof}
This follows since we are working modulo the ideal of bounded subsets of $A$. Indeed, let $\cal{F}=\lb f_\al:\al<\rho\rb$, and we will show that $\cal{F}$ witnesses the result. Note that if $Z\in[\rho]^\rho$, then $\la f_\al:\al\in Z\ra$ is also a scale. But then $\operatorname{ub}(\lb f_\al:\al\in Z\rb)$ must be unbounded in $\sup(A)$, as otherwise we contradict that $\la f_\al:\al\in Z\ra$ is cofinal in $\prod A$ modulo the bounded ideal.
\end{proof}

Now we examine one way in which $\spec^*(A)$ can play a traditional PCF-theoretic role. We begin with some background: a remarkable phenomenon in PCF theory is that scales of length $\mu^+$ (where $\mu$ is singular) can be used to ``lift" a property which holds at the $\mu_i$ to hold at $\mu^+$. For example, Shelah proved that the failure of being a J{\'o}nsson cardinal lifts in this way; we will discuss this in more detail in a moment. Other examples include Theorem 3.5 of \cite{CanonicalStructure1} and a theorem of Todorcevic about lifting the failure of certain square bracket partition relations (\cite{TodorcevicPartitioning}).

Here we include a very short review of the notion of a J{\'o}nsson cardinal, referring the reader to \cite{EisworthHB} for more details.

\begin{definition}\hfill
\begin{enumerate}
    \item An \emph{algebra} is a structure $\cal{A}=\la A,f_n\ra_{n<\om}$ so that each $f_n$ is a finitary function mapping $A$ to $A$.
    \item A \emph{J{\'o}nsson algebra} is an algebra without a proper subalgebra of the same cardinality.
    \item A cardinal $\lam$ is said to be a \emph{J{\'o}nsson cardinal} if there does \textbf{not} exist a J{\'o}nsson algebra of cardinality $\lam$.
\end{enumerate}
\end{definition}

J{\'o}nsson cardinals can be characterized in terms of a coloring relation.

\begin{fact}
$\lam$ is a J{\'o}nsson cardinal iff for any $F:[\lam]^{<\om}\to\lam$, there exists an $H\in[\lam]^\lam$ so that the range of $F\res [H]^{<\om}$ is a proper subset of $\lam$.
\end{fact}

We use $[\lam]\to[\lam]^{<\om}_\lam$ to denote the coloring property from the previous fact. We can also characterize this in terms of elementary submodels. The next item is almost exactly Lemma 5.6 from \cite{EisworthHB}; we have added a parameter to the statement, which does not change the proof. In the statement of the lemma, $<_\chi$ denotes a wellorder of $H(\chi)$.

\begin{lemma}\label{lemma:notJonsson}
The following two statements are equivalent:
\begin{enumerate}
    \item $\lam$ is a J{\'o}nsson cardinal.
    \item For every sufficiently large regular $\chi>\lam$, every cardinal $\ka$ so that $\ka^+<\lam$, and every parameter $P\in H(\chi)$, there is an $M\prec \la H(\chi),\in,<_\chi,P\ra$ so that
    \begin{enumerate}
        \item $\lb\lam,\ka\rb\in M$;
        \item $|M\cap\lam|=\lam$;
        \item $\lam\not\subseteq M$; and
        \item $\ka+1\seq M$.
    \end{enumerate}
\end{enumerate}
\end{lemma}

Shelah (\cite{ShelahJonsson}) has proven the following remarkable theorem:

\begin{theorem} (Shelah) Suppose that $\mu$ is singular and that $(\vec{\mu},\vec{f})$ is a scale (modulo the ideal of bounded sets) of length $\mu^+$. Additionally, suppose that each $\mu_i$ carries a J{\'o}nsson algebra (i.e., is not a J{\'o}nsson cardinal). Then $\mu^+$ carries a J{\'o}nsson algebra. 
\end{theorem}

Here we show that it suffices to assume that $\mu^+$ is in the strong part of the Tukey spectrum of $A$, provided the order type of $A$ is not too high.

\begin{theorem}\label{thm:specLift}
Suppose that $A$ is a set of regular cardinals with $\operatorname{ot}(A)<\mu:=\sup(A)$ so that every $a\in A$ carries a J{\'o}nsson algebra, and suppose that $\mu^+\in\spec^*(A)$. Then $\mu^+$ carries a J{\'o}nsson algebra.
\end{theorem}
\begin{proof}
We will show that $\mu^+$ carries a J{\'o}nsson algebra by showing that (2) of Lemma \ref{lemma:notJonsson} is false.

Fix a large enough regular cardinal $\chi$. Letting $\mu^+$, $|\operatorname{ot}(A)|$, and $A$ play the respective roles of $\lam$, $\ka$, and $P$ in (the negation of) Lemma \ref{lemma:notJonsson}(2), fix an arbitrary $M\prec \la H(\chi),\in,<_\chi,A\ra$ so that $\mu^+\in M$, $|M\cap\mu^+|=\mu^+$, and $|\operatorname{ot}(A)|+1\seq M$. We will show that $\mu^+\seq M$. Observe that $A\seq M$, since $M$ contains a bijection from $|\operatorname{ot}(A)|$ onto $A$, and since $|\operatorname{ot}(A)|+1\seq M$.

Applying the elementarity of $M$, we may find a set $\cal{F}\seq\prod A$ of functions witnessing $\mu^+\in\spec^*(A)$ with $\cal{F}\in M$. Using $\cal{F}$, we will show that there are unboundedly-many $a\in A$ so that $|M\cap a|=a$. The upshot of this is that for each such $a$, since $a$ is not a J{\'o}nsson cardinal, $a\seq M$. Since there are unboundedly-many such $a$, we conclude that $\sup(A)=\mu\seq M$. And finally, since $|M\cap\mu^+|=\mu^+$, we can conclude that $\mu^+\seq M$.

To show the existence of unboundedly-many such $a$, let $\cal{F}_M:=\cal{F}\cap M$. Since $|M\cap\mu^+|=\mu^+$, we know that $\cal{F}_M\in[\cal{F}]^{\mu^+}$. Since $\mu^+\in\spec^*(A)$, the set $\operatorname{ub}(\cal{F}_M)$ is unbounded in $A$.

Now let $a\in\operatorname{ub}(\cal{F}_M)$. Then for all $f\in\cal{F}_M$, $f(a)\in M$, since $f$ and $a$ are each members of $M$. Since $\lb f(a):f\in\cal{F}_M\rb$ is unbounded in $a$ (by definition of $a$ being an unbounded coordinate) and a subset of $M$, we conclude that $M\cap a$ has size $a$. This completes the proof.
\end{proof}

We close this section by providing a bound on the strong part of the Tukey spectrum. First note that it follows almost immediately from the definitions that 
$$
\sup(\spec(A))\leq\cf(\prod A,<).
$$
Now let $J_{\text{bd}}$ denote the ideal of bounded sets on $A$. Note that $\prod A/J_{\text{bd}}$ does have a cofinality, but it needn't have a \emph{true} cofinality (i.e., a linearly-ordered, cofinal subset).

\begin{proposition}
Let $A$ be a set of regular cardinals, and let $\lam\in\spec(A)$. Then either $\lam\leq\cf(\prod A/J_{\text{bd}})$ or $\lam\in\spec(\bar{A})$ for some proper initial segment $\bar{A}$ of $A$.

In particular (see Definition \ref{def:strongpart}), if $\lam$ is in $\spec^*(A)$, then $\lam\leq\cf(\prod A/J_{\text{bd}})$.
\end{proposition}
\begin{proof}
Let $\mu:=\cf(\prod A/J_{\text{bd}})$, and let $\la f_\al:\al<\mu\ra$ enumerate a set of functions which is cofinal (but not necessarily increasing) in $\prod A/J_{\text{bd}}$. Suppose that $\lam>\mu$. Let $\la h_\xi:\xi<\lam\ra$ enumerate a set $\cal{F}$ of $\lam$-many functions witnessing that $\lam\in\spec(A)$. Since $\lam$ is regular and $\lam>\mu$, fix $\al<\mu$ and $X\in[\lam]^\lam$ so that for all $\xi\in X$, $h_\xi<_{J_{\text{bd}}}f_\al$. Now freeze out the tail: let $Y\in [X]^\lam$ and $\bar{a}\in A$ so that for all $\xi\in Y$ and all $a\in A\bsl\bar{A}$,
$$
h_\xi(a)<f_\al(a).
$$
Thus $\cal{F}_0:=\lb h_\xi:\xi\in Y\rb$ is bounded on every coordinate in $A\bsl\bar{a}$. Since $\cal{F}_0\in[\cal{F}]^\lam$, $\cal{F}_0$ is unbounded in $(\prod A,<)$. From this, one can argue that $\lb f_\xi\res (A\cap\bar{a}):\xi\in Y\rb$ has size $\lam$ and witnesses that $\lam\in\spec(A\cap\bar{a})$.

For the ``in particular" part of the proposition, note that if $\lam>\cf(\prod A/J_{\text{bd}})$, then the previous argument shows that there is an $\cal{F}_0\in[\cal{F}]^\lam$ so that $\operatorname{ub}(\cal{F}_0)$ is bounded in $A$.
\end{proof}

\section{Questions}

Here we record a few questions which we find interesting. The first question restates Question \ref{q:theQ}, the main one driving this line of research:

\begin{question}\label{Q:main}
 Does $\zfc$ prove that for all sets $A$ of regular cardinals, $\pcf(A)=\spec(A)$?
\end{question}

A restricted version of Question \ref{Q:main}, to be read in light of Theorem \ref{thm:GM}, is this:

\begin{question}
Does $\zfc$ prove that for all \emph{progressive} sets $A$ of regular cardinals, $\pcf(A)=\spec(A)$?
\end{question}

One can also ask about the relationship between $\spec^*(A)$ and $\pcf(A)$, as in the next two questions:

\begin{question}
Does $\zfc$ prove that $\spec^*(A)\seq\pcf(A)$?
\end{question}

\begin{question}
Does $\zfc$ prove that if $A$ is a set of regular cardinals without a max, then $\spec(A)\bsl\sup(A)\seq\spec^*(A)$?
\end{question}

The following question should be read in light of Theorem \ref{thm:limpcf}:

\begin{question}
Does $\zfc$ prove that for all sets $A$ of regular cardinals, $$
\spec(A)\seq\pcf(A)\cup\lim(\pcf(A))?
$$
\end{question}

The next question connects to Theorem \ref{thm:notspec} and Corollary \ref{cor:notspec}:

\begin{question}
Is a weakly compact necessary to get a model in which $\ka$ is a regular limit, $A\seq\ka$ is unbounded and non-stationary, and $\ka\notin\spec(A)$?
\end{question}

Theorem \ref{thm:specLift} showed that the strong part of the Tukey spectrum can be used in place of PCF-theoretic scales to lift the property of not being a J{\'o}nsson cardinal. Where else, if at all, can $\spec^*(A)$ be used in this way? In particular, we ask whether the strong part of the Tukey spectrum is enough to generalize a classic result of Todorcevic (\cite{TodorcevicPartitioning}; see the treatments in \cite{BurkeMagidor} and \cite{EisworthHB}).

\begin{question}
Suppose that $A$ is a set of regular cardinals cofinal in a singular $\mu$ so that every $\ka\in A$ fails to satisfy $\ka\to [\ka]^2_\ka$. Suppose that $\mu^+\in\spec^*(A)$. Does this imply that $\mu^+$ fails to satisfy $\mu^+\to[\mu^+]^2_{\mu^+}$?
\end{question}

\bibliographystyle{plain}
\bibliography{Gilton_SpecPCF}

\end{document}